\documentclass[12pt]{amsart}

\usepackage{amssymb}
\usepackage{color,ulem}

\usepackage{enumerate}

\makeatletter
\@namedef{subjclassname@2010}{%
  \textup{2010} Mathematics Subject Classification}
\makeatother

\newtheorem{theorem}{Theorem}[section] 
\newtheorem{lemma}[theorem]{Lemma}
\newtheorem{corollary}[theorem]{Corollary}
\newtheorem{proposition}[theorem]{Proposition}

\theoremstyle{definition}
\newtheorem{definition}[theorem]{Definition}
\newtheorem{remark}[theorem]{Remark}
\newtheorem{example}[theorem]{Example}

\numberwithin{equation}{section}

\usepackage{hyperref}

\def \ff{f}

\def \pN{\mathcal{N}}
\def \NN{\mathbb{N}}
\def \sp{p^*}
\def \tp{\tilde{p}}
\def \RR{\mathbb{R}}
\def \Rd{{\RR^d}}
\def \chiI{{\bf 1}}

\newcommand{\new}[1]{{\textcolor{black}{#1}}}
\newcommand{\newM}[1]{{\textcolor{black}{#1}}}

\begin{document}

\baselineskip=17pt

\title[Gaussian estimates for Schr{\"o}dinger perturbations]{Gaussian estimates for Schr{\"o}dinger perturbations}

\author[K. Bogdan]{Krzysztof Bogdan}
\address{Krzysztof Bogdan\\Institute of Mathematics of Polish Academy of Sciences\\
Institute of Mathematics and Computer Science, Wroc{\l}aw University of Technology\\
Poland}
\email{bogdan@pwr.wroc.pl}

\author[K. Szczypkowski]{Karol Szczypkowski}
\address{Karol Szczypkowski\\Institute of Mathematics and Computer Science\\ Wroc{\l}aw University of Technology, Wyb. Wyspia\'nskiego 27, 50-370 Wroc{\l}aw,
   Poland}
\email{karol.szczypkowski@pwr.wroc.pl}

\date{}

\begin{abstract}
We propose a new general method of estimating Schr{\"o}dinger perturbations of transition densities using an auxiliary transition density as a majorant of the perturbation series. We present applications to Gaussian bounds by proving
an optimal inequality involving four Gaussian kernels, which we 
call 4G Theorem. The applications come with 
honest control of constants in estimates of Schr\"odinger perturbations of Gaussian-type heat kernels and also allow for specific non-Kato perturbations.
\end{abstract}



\thanks{The second author was supported by NCN (grant no. 2011/03/N/ST1/00607) and MNiSW (grant no. IP2012 018472).}

\maketitle

\section{Introduction and main results}
A Schr\"odinger perturbation is an addition of an operator of multiplication to a given operator. On the level of inverse operators, the addition results in 
\new{a} resolvent or Duhamel's or
\new{a} perturbation formula, and under certain conditions it yields von Neumann or perturbation series for the inverse of the perturbation. The subject is very wide, and we intend to 
touch 
it in the case when the inverse operator is an evolution semigroup, in fact, a transition density. In this case a convenient and simple 
setup is that of integral operators on space-time, and the perturbation series has an exponential flavor due to repeated integrations on time simplexes.
In this work we propose a general method for pointwise estimates of the series, and we demonstrate its versatility by estimating transition densities of Schr\"odinger perturbations of heat kernels on $\Rd$.

In an earlier work, Bogdan, Jakubowski and Sydor \cite{MR3000465}
developed a technique {for {\it sharp} pointwise estimates of} Schr\"odinger perturbations $\tp$ of transition densities $p$ 
and more general integral kernels on
\new{a} state space $X$ by functions $q\geq 0$. The method rests on the assumption 
\begin{equation}\label{eq:kk_1}
\int_s^t\int_X p(s,x,u,z)q(u,z)p(u,z,t,y)dzdu\leq [\eta +Q(s,t)]p(s,x,t,y),
\end{equation} 
where $s<t$, $x,y \in X$, $0\le \eta<\infty$, and $0\le Q(s,u)+Q(u,t)\le Q(s,t)$ and $dz$ is a measure on $X$.
The left-hand side of \eqref{eq:kk_1} defines the term $p_1(s,x,t,y)$ in the perturbation series,
\begin{equation}\label{def:tp}
\tp(s,x,t,y)=\sum_{n=0}^{\infty}p_n(s,x,t,y)\,,
\end{equation}
(see Section~\ref{sec:Gen_est} for detailed definitions),
and so $p_1(s,x,t,y)/p(s,x,t,y)\le \eta+Q(s,t)$. 
The bound is u\-ni\-form in space and locally uniform in time,
and it propagates as follows, 
\begin{align}\nonumber
p_n(s,x,t,y)&\leq p_{n-1}(s,x,t,y)\left[\eta + \frac{Q(s,t)}{n}\right]\\
\label{eq:gpte2}
&\leq p (s,x,t,y)\prod_{k=1}^{n}\left[\eta + \frac{Q(s,t)}{k}\right].
\end{align}
Furthermore, if $0<\eta <1$, then \eqref{def:tp} and \eqref{eq:gpte2} yield
\begin{equation}\label{eq:metgKpt}
\tilde{p} (s,x,t,y)\leq p (s,x,t,y){\left(\frac{1}{1-\eta}\right)}^{1+Q(s,t)/\eta},
\end{equation}
and if $\eta =0$, then 
\begin{equation}\label{eq:gpte4}
\tilde{p}(s,x,t,y)\leq p (s,x,t,y)e^{Q(s,t)}.
\end{equation}
The above estimates are sharp\new{,} i.e.\new{,} the ratio of the upper bound and (the trivial lower bound) $p$ is bounded locally in time. In fact, as shown by Bogdan, Hansen and Jakubowski \cite[Example 4.3 and 4.5]{2011-BHJ} the exponential factors in \eqref{eq:metgKpt} or \eqref{eq:gpte4}
are very nearly optimal.
The estimates also apply to 
rather general integral kernels on space-time, without assuming Chapman-Kolmogorov equations.
It is now crucial 
to verify 
\eqref{eq:kk_1} for given $p$ and $q$. To this end
we usually try to split and estimate the singularities 
of the integrand $(u,z)\mapsto p(s,x,u,z)q(u,z)p(u,z,t,y)$ in \eqref{eq:kk_1}.
This is 
straightforward
if $q$ satisfies a suitable Kato-type condition and 
\new{the} 3G Theorem holds for $p$ (see  \cite[Remark~2]{MR3000465}, \eqref{3P:ineq} and the discussion in Section~\ref{sec:Gen_est} and Section~\ref{sec:Gaus_tr} below). 
The latter is the case, e.g.,  for the transition density of the fractional Laplacian as described by Bogdan, Hansen and Jakubowski \cite[Corollary 11]{MR2457489}, \cite[Example~4.13]{2011-BHJ},
and for the potential kernel of the stable subordinator \cite[Example~2]{MR3000465}, corresponding to the fact that the functions have power-type asymptotics.
However, due to their exponential decay, 3G fails (see \eqref{no3G})  for the Gaussian kernels
\begin{align}\label{def:g_c}
g_a (s,x,t,y)= [4\pi (t-s)/a]^{-d/2} \exp \big\{-|y-x|^2/[4(t-s)/a]\big\}.
\end{align}
Here {$d\in \NN$}, $a>0$, $s<t$, $x,y\in \Rd$, and we let $g_a (s,x,t,y)=0$ if $s\ge t$.
We observe that for $0<a<b<\infty$ we have
\begin{align}
g_b(s,x,t,y)&\leq (b/a)^{d/2}g_a (s,x,t,y)\,. \label{ineq:gbga}
\end{align}
Motivated by these observations, the results of Zhang \cite{MR1488344,MR1978999} and the arguments of  Jakubowski and Szczypkowski
\cite{MR2643799,MR2876511}, we propose 
a  more flexible method of estimating Schr\"odinger perturbations of {\it transition densities} on $X$ with respect to a measure $dz$.
The method employs an auxiliary transition density $p^*$
as an approximate {\it majorant}
of $p$ substituting for $p(\cdot,\cdot,t,y)$ in \eqref{eq:kk_1}. Namely, we assume that two measurable transition densities satisfy
\begin{equation}\label{eq:pCtp}
p(s,x,t,y) \leq C \sp (s,x,t,y)\,,
\end{equation}
with a constant $C\ge 1$. In addition to superadditivity of $Q(s,t)\ge 0$, we also assume that it is right-continuous  in $s$ and left-continuous in $t$, and that the following inequality holds
\begin{equation}\label{def:coeNs}
\int_s^t \int_{X} p(s,x,u,z)q(u,z) \sp(u,z,t,y)\,dz\,du \leq \big[ \eta + Q(s,t)\big] \sp(s,x,t,y)\,.
\end{equation}
We write these conditions in short as $q \in \pN(p,\sp,C,\eta,Q)$ (see Definition~\ref{def:N} below for details). They allow
to recursively estimate multiple integrals involving $p$ in the perturbation series.

In Section~\ref{sec:Gen_est} we prove our first main result, which is as follows.
\begin{theorem}\label{thm:1}
If $q\in \pN(p,\sp,C,\eta,Q)$ and $0\le \eta< 1$, then
for all $s<t$, $x,y \in X$ and $0<\varepsilon<1-\eta$ we have
\begin{equation}\label{ineq:thm1a}
\tp (s,x,t,y)\leq \sp (s,x,t,y)\left( \frac{C}{1-\eta-\varepsilon}\right)^{1+\frac{Q(s,t)}{\varepsilon}}\,.
\end{equation}
\end{theorem}
\begin{remark}
Two natural 
choices 
are: $\varepsilon=\eta$ if $0<\eta<\frac12$, and $\varepsilon=\frac{1-\eta}{2}$. 
\end{remark}
In the second part of the paper we test our methods
against Gaussian-type estimates.
To this end we first elaborate 
\cite[inequality (4.4)]{MR1488344}
by giving the best constant in the following estimate 
involving four different Gaussian kernels (hence 4G). 
\begin{theorem}[4G]\label{thm:4P}
\new{For $\alpha>0$, let}
$L(\alpha) = \max\limits_{\tau\geq \alpha\vee 1/\alpha } \left[  \ln\left(1+\tau\right)  -\frac{\tau-\alpha}{1+\tau} \ln (\alpha \tau )\right]$, \new{and let} $0<a<b<\infty$ and
$M=\left(\frac{b}{ b-a}\right)^{d/2}\exp\left[\frac{d}{2}L(\frac{a}{b-a})\right]$.
Then 
\begin{equation}\label{eq:4G}
\!\!\!\!g_b(s,x,u,z)g_a(u,z,t,y)\!\leq\! M\!\left[g_{b-a}(s,x,u,z)\vee g_a(u,z,t,y)\right]g_a(s,x,t,y),
\end{equation}
for all $s<u<t$ and $x,z,y\in \Rd$.
This fails for some $s<u<t$, $x,z,y\in \Rd$, if $M<\left(\frac{b}{b-a}\right)^{d/2}\exp\left[\frac{d}{2}L(\frac{a}{b-a})\right]$.
\new{Further,} we have
$\left(\frac{b}{ b-a}\right)^{d/2}\exp\left[\frac{d}{2}L(\frac{a}{b-a})\right]=\left( 1-a/b \right)^{-d}$
if 
$
1/(1+e^{-1/2})\le a/b<1$.
\end{theorem}

The proof of Theorem~\ref{thm:4P} is given
in Section~\ref{sec:eGk}.

Then, in Section~\ref{sec:Gaus_tr} 
we obtain precise  Gaussian estimates for the fundamental solution of 
Schr\"odinger perturbations of 
second order 
parabolic differential operators, recovering and improving existing results, which we 
discuss there at some length.
They follow by considering $g_b$ and $g_a$ of Theorem~\ref{thm:4P} 
as (multiples of)
$p$ and $p^*$ of Theorem~\ref{thm:1}. 
Most of our discussion in Section~\ref{sec:Gaus_tr} is summarized in the following theorem on Borel measurable transition densities  $p$ on $X=\Rd$ with the Lebesgue measure $dz$. To simplify
\new{the} notation, for $d\ge 3$ and $U\colon\Rd\to\RR$ we denote
\begin{align}\label{def:I}
I_{\delta}(U)= \sup_{x\in\Rd} \int_{|z-x|
< \delta}  \frac{|U(z)|}{|z-x|^{d-2}}\,dz\,, \quad \delta>0\,,
\end{align}
and we let $c_0=c_0(d)=\Gamma(d/2-1)\pi^{-d/2}/4$.
\begin{theorem}\label{thm:new}
Let $d\geq 1$. Assume that 
$b>0$, $\Lambda\ge 1$ and $\lambda\in\RR$ exist such that\new{,} for $s<t$,
\begin{equation}\label{eq:Gub}
p(s,x,t,y)\le \Lambda e^{\lambda(t-s)}g_b(s,x,t,y),\quad x,y\in \Rd\,.
\end{equation}
Let $0<a<b$ and $C=\Lambda(b/a)^{d/2}$.
If $q\in \pN(g_b,g_a,(b/a)^{d/2}, \eta,Q)$, then for all $s<t$, $x,y\in\Rd$ and $0<\varepsilon <1-\Lambda\eta$,
\begin{align}
\label{eq:geKa}
\tp(s,x,t,y) &
\leq 
\left( \frac{C}{1-\Lambda\eta-\varepsilon}\right)^{1+\frac{\Lambda Q(s,t)}{\varepsilon}}e^{\lambda(t-s) }g_a(s,x,t,y)\,.
\end{align}
\new{Further,} if $d\geq 3$, $h>0$, $q$ is time-independent and $I_{\sqrt{h}}(q)<\infty$, then $q\in\pN(g_b,g_a,(b/a)^{d/2},\eta,Q)$ with
\begin{align*}
\eta= b\, c_0 M\, I_{\sqrt{h}}(q)\,,\qquad
Q(s,t)=(t-s) I_{\sqrt{h}}(q)\,2M/( |B(0,1/2)|h)\,.
\end{align*}
\end{theorem}
\noindent
\newM{For more details on the role of the term $I_{\delta}(q)$ see \eqref{eq:psup}.} Here $|B(0,1/2)|$ is the volume of the ball with radius $1/2$,
$M$ is the optimal constant from \eqref{eq:4G},
and the smallness of $\eta$ may follow from having $b$ small and $a$ proportional to $b$ or choosing $h$, hence $I_{\sqrt{h}}(q)$, small. 
This brings about
honest control of constants
in estimates, which is not available by other existing methods. Our bounds 
of $\tp$ are automatically global in time, and we do not need to patch together estimates obtained in small time intervals by means of Chapman-Kolmogorov equations. 
Noteworthy, our methods are not restricted to Gaussian-type kernels.
Further 
applications, e.g. to perturbations of the transition density of the $1/2$-stable subordinator, will be given in a forthcoming paper.

In Section~\ref{sec:Gaus_tr} we also describe connections to second order differential operators and we identify some of the transition densities $\tp$ given by \eqref{def:tp} as left inverses of second order differential operators on space-time: for $s\in \RR$, $x\in\Rd$ and $\phi \in C_c^{\infty}(\RR\times\Rd)$,
\begin{align*}
&\int_s^{\infty}\!\! \int_{\Rd}\!\! \!\!\tp (s,x,u,z) \\
&\!\!\left(\!\! \frac{\partial }{\partial u} + 
\sum_{i,j=1}^n a_{ij}(u,z)\frac{\partial^2}{\partial z_i\partial z_j}+
\sum_{i=1}^n b_i(u,z)\frac{\partial }{\partial z_i}
+ q(u,z)\!\!\right)\!\!\phi(u,z) \,dzdu=-\phi(s,x)\,.
\end{align*}

In Section~\ref{sec:Misc} we give miscellaneous methodological comments on superadditivity of $Q$. 

Our inspiration mainly comes from \cite{MR2876511} and \cite{MR3000465}. Our ideas are also similar to those developed for Gaussian estimates in \cite{MR1488344}.
In particular, the condition \eqref{def:coeNs} for our main Theorem~\ref{thm:1} may be considered as a generalization of the Main Lemma~4.1 of \cite{MR1488344}, 
while the 4G inequality in Theorem~\ref{thm:4P} yields an alternative, synthetic justification of 
that lemma. 
Furthermore, the proof of \cite[inequality (4.4)]{MR1488344}, yields 4G, except for the optimal constant $M$.
It is thus of interest that the approach of \cite{MR1488344}, which was taylor-made for the Gaussian kernel, has a more general context given by Theorem~\ref{thm:1}.

\section{Estimates for general transition densities}\label{sec:Gen_est}

Let $X$ be an arbitrary set with a $\sigma$-algebra $\mathcal{M}$ and a (non-negative) $\sigma$-finite measure $m$ defined on $\mathcal{M}$. To simplify the notation we will write $dz$ for $m(dz)$ in what follows. We also consider the Borel subsets $\mathcal{B}$ of $\RR$, and the Lebesgue measure, $du$, defined on $\RR$. The {\it space-time}, $\RR\times X$, will be equipped with the $\sigma$-algebra $\mathcal{B}\times\mathcal{M}$ and the product measure $du\,dz=du\,m(dz)$.

We will consider a {\it measurable transition density} $p$ on space-time, i.e.\new{,}
we assume that $p: \RR\times X \times\RR\times X\to [0,\infty]$ 
is 
$\mathcal{B}\times\mathcal{M}\times\mathcal{B}\times\mathcal{M}$-measurable
and the Chapman-Kolmogorov equations hold for all $x,y \in X$  and $s<u<t$:
\begin{align}\label{assume1}
\int_{X} p(s,x,u,z)p(u,z,t,y)\,dz = p(s,x,t,y)\,.
\end{align}
Let $q\colon \RR \times X \to [0,\infty]$ 
be (nonnegative and) $\mathcal{B}\times\mathcal{M}$-measurable.
(All the functions considered below are tacitly assumed measurable on their respective domains.)
The Schr{\"o}dinger perturbation $\tilde p$ of $p$ by $q$ is defined by the series \eqref{def:tp},
where $p_{0}(s,x,t,y)=p(s,x,t,y)$,
\begin{equation*}
p_1(s,x,t,y)=\int_s^t\int_X p(s,x,u,z)q(u,z)p(u,z,t,y)dzdu,
\end{equation*}
and for $n=2,3,\ldots$,
\begin{align}
&p_n(s,x,t,y) = \int_s^t \int_{u_1}^t \ldots \int_{u_{n-1}}^t \;\int_{(\Rd)^n}
p(s,x,u_1,z_1)q(u_1,z_1)
 \nonumber
\\
&p(u_1,z_1,u_2,z_2)\cdots q(u_n,z_n)p(u_n,z_n,t,y)dz_n\cdots dz_1 du_n \cdots du_1. \label{defpn}
\end{align}
By Fubini-Tonelli, for $n=1,2,\ldots$, we have
\begin{align}\label{def:p_n}
p_n(s,x,t,y)&=\int_s^t \int_{X} p(s,x,u,z)q(u,z) p_{n-1}(u,z,t,y)dzdu\,,
\end{align} 
and
\begin{align}\label{def:p_n2}
p_n(s,x,t,y)&=\int_s^t \int_{X} p_{n-1}(s,x,u,z)q(u,z) p(u,z,t,y)dzdu\,.
\end{align}
By \cite[Lemma 1]{MR2457489},
for all $s<u<t$, $x,y\in X$ and $n\in \NN_0=\{0,1,\ldots\}$, 
\begin{align}\label{eq:pCK}
\sum_{m=0}^n \int_{X} p_m (s,x,u,z)  p_{n-m} (u,z,t,y)\,dz = p_n (s,x,t,y)\,.
\end{align} 
By \cite[Lemma 2]{MR2457489}, Chapman-Kolmogorov equations hold for $\tp$.
Clearly, $\tp\ge p$.
\begin{remark}\label{sper}
The perturbation,
say 
$p_V$, is given by the same formulae if $V:\RR\times \Rd\to \mathbb{C}$ (takes on complex, in particular negative values), 
provided $p_{|V|}$ is finite. Indeed, $p_V$ then converges absolutely and
\begin{equation}
|p_V|\leq p_{|V|}.
\end{equation}
A detailed discussion of signed real-valued perturbations  
of transition densities 
is given in \cite{MR2457489}, 
with a positive lower bound for $p_V$ resulting from Jensen's inequality.
A probabilistic interpretation of $p_n$ and $p_V$ may also be found in \cite{MR2457489}.
\end{remark}

Below we focus on {upper bounds} of $\tp=p_q$ for $q\ge 0$ and {\it transition densities} $p$, as defined above.
This is a less general setting than that of \cite{MR3000465},
but within this setting our bound \eqref{ineq:thm1a}
holds under 
more flexible condition \eqref{def:coeNs} on $p$ and $q$.
Namely, we consider 
another (measurable) transition density $\sp$ 
and $C\geq 1$ such that for all $x,y\in X$ and  $s<t$, inequality \eqref{eq:pCtp} holds.
We can estimate the cumulative effect of the integrations involved in \eqref{defpn}, or (\ref{def:p_n}).
The following result is an analogue of \cite[Lemma 5]{MR2876511} and \cite[Example 4.5]{2011-BHJ}. 
\begin{lemma}\label{lem:2}
Let $\theta \ge 0$
and $s_0< \ldots < s_k=t$ 
be such that 
\begin{equation}\label{eq:spq}
\int_s^{s_{i+1}} \int_{X} p(s,x,u,z)q(u,z) \sp(u,z,s_{i+1},y)\,dzdu \leq \theta\, \sp (s,x,s_{i+1},y), 
\end{equation}
for all $i=0,\ldots,k-1$, $s\in [s_i,s_{i+1}]$ and $x,y\in X$. Then for every $n\in \NN_0$,
\begin{equation}\label{eq:epkn}
p_n(s,x,t,y)\leq \binom{n+k-1}{k-1} \theta^n C^{k} \sp(s,x,t,y),\quad s\in [s_0, s_1],\,x,y\in X.
\end{equation} 
\end{lemma}

\begin{proof}
For $k=1$, the estimate holds for $n=0$ by (\ref{eq:pCtp}), and then it holds for all $n\geq 1$ by induction, (\ref{def:p_n}) and (\ref{eq:spq}):
\begin{align*}
p_n(s,x,t,y)&=\int_s^t\int_X p(s,x,u,z)q(u,z)p_{n-1}(u,z,t,y)dzdu\\
&\le \theta^{n-1}C \int_s^t\int_X p (s,x,u,z)q(u,z)\sp(u,z,t,y)dzdu\\
&\le \theta^n C\sp(s,x,t,y),\quad \mbox{ where }\quad  s\in [s_0,s_1],\quad x,y\in X.
\end{align*}
If $k\ge 2$,
then by (\ref{eq:pCK}), induction and Chapman-Kolmogorov for $\sp$,
\begin{align*}
&p_n(s,x,t,y) = \sum_{m=0}^n \int_{X} p_m (s,x,s_{k-1},z)p_{n-m} (s_{k-1},z,t,y)\,dz \\
&\leq \sum_{m=0}^n \int_{X} \binom{m+k-2}{k-2} \theta^m C^{k-1} \sp (s,x,s_{k-1},z)   \theta^{n-m} C  \sp (s_{k-1},z,t,y)\,dz \\
&= \binom{n+k-1}{k-1} \theta^n C^{k} \sp(s,x,t,y)\,,\quad \mbox{ if } s\in [s_0,s_1],\,x,y\in X, n\in \NN_0.
\end{align*}
\end{proof}

In passing we note that the assumption and conclusion in the statement of \cite[Lemma 5]{MR2876511} need a slight strengthening for the induction to work properly:
each $t_{i+1}$ in the assumption there should be replaced by $\tau$ in
$[t_i,t_{i+1}]$, and each $t$ in the conclusion should be replaced by $\tau$
in $[t_{i},t_{i+1}]$ (then one proceeds as in the proof of Lemma~\ref{lem:2} above).
The correction does not influence applications of Lemma~5 or other results in \cite{MR2876511}.

{We further let $Q \colon \RR \times \RR \to [0,\infty)$ be {\it regular superadditive}, meaning that
\begin{equation}\label{ineq:Q}
Q(s,u) + Q(u,t) \leq Q(s,t)\,, \qquad \mbox{if}\quad s< u< t \,,
\end{equation}
$Q(s,t)=0$ if $s\geq t$, $s\mapsto Q(s,t)$ is right-continuous and $t\mapsto Q(s,t)$ is left-continuous. (The continuity assumptions 
are rather innocuous,
as we explain later on in Lemma~\ref{lem:3} and Lemma~\ref{lem:QQ}). We see that $t\mapsto Q(s,t)$ is non-decreasing and $s\mapsto Q(s,t)$ is non-increasing.
 For instance, if $\mu$ is a Radon measure on $\RR$, then $Q(s,t)=\mu(\{u\in \RR: s< u <t\})$ is regular superadditive}.
A regular superadditive $Q$ is infinitely decomposable in the following sense.

\begin{lemma}\label{lem:4}
Let $s\leq t$, $k\in \NN$ and $\theta\ge 0$ be such that $Q(s,t)\leq k\theta$. Then
$s=s_0\le s_1\le \ldots \le s_k =t$ exist such that $Q(s_{i-1},s_{i})\leq \theta$ for $i= 1,\ldots ,k$.
\end{lemma}
\begin{proof}
We may and do assume that $k>1$ and $(k-1)\theta<Q(s,t)\leq k\theta$.
Let $s_i=\inf \{u: Q(s,u)\geq i\theta\}$ for $i=1,\ldots, k-1$.
If $s\le u <s_i$, then $Q(s,u)<i\theta$,  and so $Q(s,s_i)\leq i\theta$.
If $s_i<u<s_{i+1}$, then $Q(s,u)\geq i\theta$ and 
$Q(s,u)+Q(u,s_{i+1})\leq Q(s,s_{i+1})\leq (i+1)\theta$, thus $Q(u,s_{i+1})\leq \theta$. Letting $u\to s_i$ we obtain $Q(s_i,s_{i+1})\leq \theta$, which is also true if $s_i=s_{i+1}$.
 \end{proof}

\begin{definition}\label{def:N}
We write $q \in \pN(p,\sp,C,\eta,Q)$ {if} $q\geq 0$ is defined (and measurable) on space-time, $p$ and $\sp$ are (measurable) transition densities, $C\geq 1$, $\eta\geq 0$, $Q$ is regular superadditive, and {\rm (\ref{eq:pCtp})} and  \eqref{def:coeNs} 
hold
for all $s<t$ and $x,y\in X$.
\end{definition}

The terms $\eta$ and  $Q(s,t)$ of \eqref{def:coeNs} propagate differently in estimates of $p_n$ below. We may think about $\eta$ as giving a bound for instantaneous growth of mass, while $Q$ gives a cap for growth accumulated over time (see \cite{MR3000465} and \cite{2011-BHJ} for such insights).

We are in a position to prove our first main result.
\begin{proof}[Proof of Theorem~\ref{thm:1}]
Let 
$k\in \NN$. By Lemma \ref{lem:4}, $s_0=s<s_1<\ldots <s_k =t$ exist such that $Q(s_{i-1}, s_{i})\leq Q(s,t)/k$ for $i=1,\ldots,k$. \new{For $\varepsilon\in (0,1-\eta)$ we choose $k\in\NN$ such that $(k-1)\varepsilon\le Q(s,t)<k\varepsilon$.} By Lemma \ref{lem:2}  with $\theta = \eta+Q(s,t)/k$,
and by Taylor\rq{}s expansion, 
for all $x,y \in X$ we get
\begin{align*}
\tp&(s,x,t,y)\leq \sum_{n=0}^{\infty}
\new{p_n(s,x,t,y)}\\
&\leq \sum_{n=0}^{\infty} \binom{n+k-1}{k-1}C^{k}\left[\eta+Q(s,t)/k\right]^n \sp(s,x,t,y)\\
& = \left( \frac{C}{1-\eta-Q(s,t)/k} \right)^k \sp(s,x,t,y)\new{\,.}
\end{align*}
This ends the proof.
\end{proof}

Analogous results hold if we state our assumptions and conclusions for $s,t$ in a finite time horizon:
$-\infty<t_1\le s<t \le t_2<\infty$.
If $Q$ in Theorem~\ref{thm:1} 
is 
bounded,
then $\tp \leq const.\, \sp$ 
uniformly 
in time.
We may consider $q\in \pN(p,p^*,C,\eta,Q)$ with 
\newM{$\eta <1$},
$q_1(u,z) = q(u,z) \, \chiI_{[0,1]}(u)$, and
bounded 
superadditive function 
$Q_1(s,t)=Q((s\vee 0)\land 1 , (t\land 1)\vee 0)$.
Then,
\begin{align*}
\int_s^t \int_{X} p(s,x,u,z)q_1(u,z)p^*(u,z,t,y)\,dzdu \leq \big[ \eta + Q_1 (s,t)\big] p^* (s,x,t,y)\,.
\end{align*}
Thus, by Theorem \ref{thm:1}, 
$\tilde p\le c p^*$ uniformly in time.
If $p^*$ is not comparable with $p$ in space then the estimates in Theorem~\ref{thm:1} cannot be sharp. This is regrettable, but quite common, e.g., in Schr\"odinger perturbations of Gaussian kernel discussed later on in the paper.
The role of $\sp$ is similar to that of $f$ in \cite[Theorem~3.2]{2011-BHJ}, but the results of \cite{2011-BHJ} do not 
apply in the present setting, if $p\neq \sp$.
If (\ref{def:coeNs}) holds with $\sp$ replaced by $p$, then we may take $\sp=p$ and  $C=1$ in (\ref{eq:pCtp}) and Theorem~\ref{thm:1}. 
However, in this case a more efficient inductive argument of \cite{MR3000465} gives better estimates \eqref{eq:metgKpt} and \eqref{eq:gpte4} above.
If $q(u,z)\le f(u)$, then we may take $Q(s,t)=\int_s^t f(u)du$, $\eta=0$ and $p=p^*$.
In fact, if $q(u,z)=f(u)$, then $p_n(s,x,t,y)=Q(s,t)^n p(s,x,t,y)/n!$ and
$\tp(s,x,t,y)=e^{Q(s,t)}p(s,x,t,y)$.

Less trivial applications of Theorem~\ref{thm:1} require 
detailed study of $p$ and $q$, and judicious choice of $p^*$.
In particular, to estimate 
$\tp$ for given $p$, $p^*$ and $q$, we wish to verify (\ref{def:coeNs}). This task
may be facilitated by splitting
the singularities of $p$ and $\sp$ in the integral of (\ref{def:coeNs}). 
Various versions of the 3G Theorem are used to this end,
see \cite{MR2457489}, Bogdan and Jakubowski \cite{MR2283957}, and the "elliptic case" of Cranston, Fabes and Zhao \cite{MR936811}, and Hansen \cite{MR2207878}.
For instance the transition density of the fractional Laplacian enjoys the following 3G inequality
\begin{equation}
        p(s,x,u,z) \land p(u,z,t,y) \le c\; p(s,x,t,y), 
       \label{3P:ineq}
\end{equation}
where $x,z,y \in  \Rd$, $s<u<t$ (\cite[Theorem~4]{MR2283957}), and this yields
\begin{align*}
  p(s,x,u,z)p(u,z,t,y)&=  [  p(s,x,u,z) \lor p(u,z,t,y)][p(s,x,u,z) \land p(u,z,t,y)] \\
&\le c\; [p(s,x,u,z) + p(u,z,t,y)]p(s,x,t,y).
\end{align*}
In this situation we can use $p^*=p$ in Theorem~\ref{thm:1} to estimate $\tilde p$, provided 
\begin{align*}
c\int_s^t \int_{X} p(s,x,u,z)q(u,z)\,dzdu +
c\int_s^t \int_{X} q(u,z)p(u,z,t,y)\,dzdu \leq 
\eta + Q(s,t)\,.
\end{align*}
Noteworthy,
3G  fails for Gaussian kernels, 
and  for such kernels the methods of \cite{MR3000465} fall short of optimal known results. This circumstance largely motivates the present development.
In the next sections we show how to estimate quite general Schr\"odinger perturbations of Gaussian kernels
by means of Theorem~\ref{thm:1}.
The application depends on the  4G inequality stated in \eqref{eq:4G}
of Theorem~\ref{thm:4P}\new{, which} 
partially substitutes for 3G. We note that Theorem~\ref{thm:4P} improves \cite[(4.4)]{MR1488344},
since we give an optimal constant in \eqref{eq:4G}. 
Explicit constants
matter in our applications, because we specifically require 
$\eta<1$
in Theorem~\ref{thm:1}.

\section{Estimates of Gaussian kernels}\label{sec:eGk}

{As usual, $a\vee b=\max\{a,b\}$ and $a\wedge b=\min\{a,b\}$.} Let $0< \alpha<\infty$, and 
\begin{align}
\label{eq:wkp}
L(\alpha)
&=\max_{\tau\geq \alpha\vee 1/\alpha } \left[  \ln\left(1+\tau\right)  -\frac{\tau-\alpha}{1+\tau} \ln (\alpha \tau )\right]
 \\
&= \max_{\tau\geq \alpha\vee 1/\alpha} \left[  \ln\left(1+\frac1{\tau}\right) -\ln \alpha +\frac{1+\alpha}{1+\tau} \ln (\alpha \tau )\right].
 \nonumber
\end{align}
Clearly, $L(\alpha)<\infty$, and $\tau=\alpha\vee 1/\alpha$ yields $L(\alpha)\ge \ln (1+\alpha\vee 1/\alpha)$. 
By an application of calculus,
$L(\alpha)=\ln(1+\alpha)$ 
if (and only if) 
$\alpha \geq e^{1/2}$.
We let
 \begin{align*}
{\ff}(\tau,x)= \ln \tau+x^2/\tau\,, \qquad \tau>0,\; x\geq 0.
\end{align*}

\begin{lemma}\label{lem:imp}
If $\alpha>0$, $L=L(\alpha)$, $\xi,\eta\geq 0$, and $\tau>0$, then
\begin{align}\label{ineq:4}
{\ff} (1+\tau, \xi +\eta)\leq \ff(1,\xi) \lor {\ff} (\alpha \tau,\eta) + \frac{\eta^2}{\tau} +L.
\end{align}
If~$L<L(\alpha)$, then the inequality fails for some $\xi,\eta\geq 0$ and $\tau>0$.
\end{lemma}
\begin{proof}
We first 
prove the following implication:
\begin{align}\label{imp:pom2}
\mbox{If } \quad \frac{\eta^2}{\alpha \tau}+ \ln(\alpha \tau) \leq \xi^2 , \quad \mbox{then} \quad\ln(1+\tau)\leq \frac{(\tau\xi - \eta)^2}{\tau(1+\tau)}  + L   \,.
\end{align}
To this end we consider 
two special cases:
\begin{enumerate}
\item[Case 1.] $\eta^2/(\alpha \tau)+ \ln(\alpha \tau) \leq \xi^2$ and  $\eta=\tau\xi$\,,
\item[Case 2.]  $\eta^2/(\alpha \tau)+ \ln(\alpha \tau) = \xi^2$ and  $\eta<\tau\xi$\,.
\end{enumerate}

Case 1 implies that $\left( \tau/\alpha - 1 \right) \xi^2 +\ln(\alpha \tau) \leq 0$. This is possible only if
$\tau\leq \alpha\vee 1/\alpha$, whence $\ln(1+\tau)\leq \ln(1+ \alpha\vee 1/\alpha)\leq L(\alpha)$, which verifies (\ref{imp:pom2}).

In Case 2, if $\tau \leq \alpha\vee 1/\alpha$, then $\ln(1+\tau)\leq \ln(1+\alpha\vee 1/\alpha)\leq L(\alpha)$ again.
 For $\tau>\alpha\vee 1/\alpha$ we consider $\xi=\xi(\eta)$ as a function of $\eta$, and we have
$$\phi(\eta):=\ln(1+\tau)-\frac{(\tau\xi-\eta)^2}{\tau(1+\tau)} \leq L .$$
Indeed, we see that the condition $\eta<\tau\xi$ holds automatically since $\eta^2/\tau^2\leq \eta^2/(\alpha \tau)=\xi^2-\ln(\alpha \tau)<\xi^2$. Our assumption now reads $\xi^2={\eta^2}/(\alpha \tau) + \ln(\alpha \tau)$, where $\xi, \eta\geq 0$ and $\tau>\alpha\vee 1/\alpha$.
Thus, $\xi\rq{}=\eta/(\alpha \tau \xi)$. Note that $\phi(0)=\ln(1+\tau)-\tau\ln(\alpha \tau)/(1+\tau)\leq L$.
Furthermore,
$$\phi\rq{}(\eta)=-2(1+\tau)^{-1}(\tau\xi-\eta)(\xi\rq{}-1/\tau).$$
We have $\phi\rq{}(\eta)=0$ only if $\xi'=1/\tau$, or $\xi=\eta/\alpha$, and then $\eta^2/(\alpha \tau)+\ln(\alpha \tau)=\eta^2/\alpha^2$ and 
$\phi(\eta)=\ln(1+\tau)-(\tau-\alpha)\ln(\alpha \tau)/(1+\tau)$.
This in fact shows that $L=L(\alpha)$ is sharp in (\ref{imp:pom2}), see \eqref{eq:wkp}.
Furthermore, $\phi^{'}(\eta) \leq 0$ if $\xi'\geq 1/\tau$, or $\left( \tau/\alpha - 1 \right)\eta^2/\alpha \geq \tau\ln(\alpha \tau)$, in particular if $\eta$ is large.
Thus, $\phi$ is decreasing for large $\eta$, which yields (\ref{imp:pom2}) in Case 2.

Consider general $\xi,\eta$ and $\tau>0$ in (\ref{imp:pom2}). If $\eta>\tau\xi$, then decreasing $\eta$ to $\tau\xi$ strengthens (\ref{imp:pom2}), so eventually we are done by Case 1.
If $\eta<\tau\xi$, then we increase $\eta$ and strengthen the consequent 
 \new{in} (\ref{imp:pom2}), 
getting 
under Case~1 or~2.

Putting \eqref{imp:pom2} differently, $\ln(1+\tau)+(\xi+\eta)^2/(1+\tau)\leq \xi^2 + \eta^2/\tau  + L$, provided $\eta^2/(\alpha \tau)+ \ln(\alpha \tau) \leq \xi^2$.
Therefore we have 
(\ref{ineq:4}) under the assumption ${\ff} (\alpha \tau,\eta) \leq {\ff}(1,\xi)$, and the constant $L$ cannot be improved.
In particular, (\ref{ineq:4}) holds if $f(1,\xi)=f(\alpha \tau,\eta)$. Decreasing $\xi$ keeps (\ref{ineq:4}) valid because $f(1+\tau,\xi+\eta)$ then decreases, too. 
\end{proof}

We note that 3G inequality fails for $g_a$ defined in \eqref{def:g_c}, because
\begin{equation}\label{no3G}
\frac{g_a(0,0,t,y)\wedge g_a(t,y,2t,2y)}{g_a(0,0,2t,2y)}= \frac{(4\pi t/a)^{-d/2} e^{-|y|^2/(4t/a)}}{(8\pi t/a)^{-d/2} e^{-|y|^2/(2t/a)}}=2^{d/2}e^{a|y|^2/4t},
\end{equation}
is not bounded in $y\in \Rd$.
The next inequality \eqref{eq:4G} between four different instances of the Gaussian kernel substitutes for 
3G, 
and so it is coined 4G. 
We note that \cite[the proof of (4.4)]{MR1488344} yields \eqref{eq:4G}, too, 
although with a 
rough constant $M$ (see also the first equality on p. 465 in \cite{MR1978999} and the last equality on p. 15 in Friedman \cite{MR0181836}).
We also acknowledge a similar result (with rough constants) for the heat kernel of smooth bounded domains 
by Riahi \cite[Lemma~3.1]{MR2320609}.
The 
optimality of the right-hand side of \eqref{eq:4G}
is important in view of \eqref{def:coeNs} and \eqref{ineq:thm1a}, and 
may be of independent interest.
In fact, inspection of our calculations also reveals that $b-a$  in $g_{b-a}$ of \eqref{eq:4G} cannot be replaced by a bigger constant.

We are in a position to prove our second main result.
\begin{proof}[Proof of Theorem \ref{thm:4P}]
We have
\begin{align}\label{eq:lGk}
\ln g_a(s,x,t,y)&=-\frac{d}{2}\ln 4\pi-\frac{d}{2}\ln(t-s)+\frac{d}{2}\ln a -\frac{a|y-x|^2}{4(t-s)}.
\end{align}
Considering $\sqrt{2d}x$, $\sqrt{2d}y$ and $\sqrt{2d}z$ instead of $x$, $y$ and $z$ in \eqref{eq:lGk}, we see that \eqref{eq:4G} is equivalent to
\begin{align*}
&-\ln(u-s)+\ln b -\frac{b|z-x|^2}{u-s}
-\ln(t-u)+\ln a -\frac{a|y-z|^2}{t-u}\\
&\le
\frac{2}{d}\ln M+
\left[-\ln(u-s)+\ln (b-a) -\frac{(b-a)|z-x|^2}{u-s}\right]\vee\\
&\left[-\ln(t-u)+\ln a -\frac{a|y-z|^2}{t-u}\right]
-\ln(t-s)+\ln a -\frac{a|y-x|^2}{t-s}.
\end{align*}
We rewrite this using the identity $a+b-a\vee b=a\wedge b$, and we obtain
\begin{align*}
&\ln b -\ln (b-a)- \frac{a|z-x|^2}{u-s}+
\left[-\ln(u-s)+\ln (b-a) -\frac{(b-a)|z-x|^2}{u-s}\right]\land\\
&\left[-\ln(t-u)+\ln a -\frac{a|y-z|^2}{t-u}\right]
\le \frac{2}{d}\ln M
-\ln(t-s)+\ln a -\frac{a|y-x|^2}{t-s}.
\end{align*}
Adding $\ln (t-u)-\ln a$ to both sides \new{(and
moving terms from one side to another)}, we have
\begin{align*}
\frac{a|y-x|^2}{t-s}&+\ln\frac{t-s}{t-u}
\le \frac{2}{d}\ln M+\ln \frac{b-a}{b}\\
&+
\frac{a|y-z|^2}{t-u}
\vee \left[\ln\frac{u-s}{t-u}+\ln \frac{a}{b-a} +\frac{(b-a)|z-x|^2}{u-s}\right]+\frac{a|z-x|^2}{u-s}.
\end{align*}
We denote $\alpha = {a}/{(b-a)}$, $\tau=(u-s)/(t-u)$, $\xi=|y-z|\sqrt{{a}/(t-u)}$ and $\eta= |z-x|\sqrt{{a}/{(t-u)}}$, and observe that $(t-s)/(t-u)=1+\tau$. Since $|y-x|\le |z-x|+|y-z|$, \new{where equality may hold,} we see that \eqref{eq:4G} is equivalent to the following inequality (to hold for all $\tau>0$ and $\xi,\eta\ge 0$),
\begin{equation}\label{eq:rtau}
\frac{(\xi+\eta)^2}{1+\tau}+\ln (1+\tau)\le
\frac{2}{d}\ln M +\ln \frac{b-a}{b}+\xi^2\vee 
\left[\frac{\eta^2}{\alpha\tau}+\ln(\alpha \tau)\right]+\eta^2/\tau.
\end{equation}
We may now use Lemma~\ref{lem:imp}. In fact, the constant $M$ in \eqref{eq:4G} is optimal if  
$$
\frac{2}{d}\ln M+\ln \frac{b-a}{b}=L(\alpha). 
$$
Considering 
$\alpha=a/(b-a)\geq e^{1/2}$,
we obtain the last statement of the theorem from a comment following \eqref{eq:wkp}.
\end{proof}

\begin{remark}\label{acb}
\rm
In applications we usually choose $a$ (smaller than but) close to $b$, so to not lose much of Gaussian asymptotics, and  in this case the optimality of the simple formula $M=\left( 1-a/b \right)^{-d}$ comes as a nice feature of our 4G Theorem.
\end{remark}

\section{Applications and discussion}\label{sec:Gaus_tr}

In this section we discuss applications of Theorem~\ref{thm:1} to 
fundamental solutions 
of second order 
parabolic differential operators.
Namely, The\-orem~\ref{thm:1}, aided by Theorem~\ref{thm:4P}, 
allows for rather singular Schr{\"o}dinger perturbations of such operators without dramatically changing the magnitude of their 
fundamental solutions.
Most of the estimates given below are known, but our proofs are more synthetic and considerably shorter,
and we have explicit 
constants in the estimates, which may be useful in homogenization problems. 
We also note that in the case of signed perturbations
(not considered here) 
very precise lower bounds 
{are obtained} from Jensen's inequality for bridges \cite{MR2457489}, see also Remark~\ref{sper} above. We begin with a discussion of 
Kato-type conditions (historical comments are given in Remark~\ref{rK}).

Let $d\geq 3$. A Borel function $U\colon \Rd \to \RR$ is of Kato class, 
if (see \eqref{def:I} for definition)
\begin{align}\label{def:Kato}
\lim_{\delta \to 0^+} I_{\delta}(U) =0 \,. 
\end{align}
A typical example is $U(z)=|z|^{-2+\varepsilon}$, where $0<\varepsilon\le2$.
By Aizenman and Simon \cite[Theorem~4.5]{MR644024}, Chung and Zhao
\cite[
Theorem~3.6]{MR1329992} or Zhao \cite[Theorem 1]{MR1132313},
(\ref{def:Kato}) holds if and only if for every $c>0$
the following condition is satisfied (see \eqref{def:g_c}),
\begin{align}\label{equiv:Kato}
\lim_{h \to 0^+} \sup_{s\in \RR,\, x\in \Rd} \int_s^{s+h} \int_{\Rd} g_c (s,x,u,z) |U(z)|\,dzdu =0\,.
\end{align}
In fact, $C_0=C_0(d,c)$ and $C_1=C_1(d,c)$ exist such that
for all $h>0$ and $U$,
\begin{align}
C_0\, I_{\sqrt{h}}(U)  
\leq \sup_{s\in\RR,\,x\in\Rd} \int_s^{s+h} \int_{\Rd} g_c (s,x,u,z)|U(z)|\,dzdu  \leq \, C_1\, I_{\sqrt{h}}(U)\,.
\label{eq:psup}
\end{align}
The lower bound of \eqref{eq:psup} is given in \cite[(4.5)]{MR644024} and \cite[Lemma~3.5]{MR1329992}. The upper bound can be proved as in \cite[Lemma~11]{MR2283957}, but for the reader's convenience we 
give a simple, explicit and more flexible argument
showing (after Proposition~\ref{prop:nice} below) that in fact in \eqref{eq:psup} we may take
\begin{equation}\label{ecK}
C_1(d,c)=\Gamma(d/2-1) \pi^{-d/2}[ c+2^d d(d-2)]/4.
\end{equation}

To this end for\new{,} $x\in \Rd$ and $r>0$\new{,} we denote $B(x,r)=\{y\in \Rd: |y-x|<r\}$, and we consider $1_{B(0,r)}$, the indicator function of the ball of radius $r>0$.
We call 
\new{$f:\Rd\to [-\infty,\infty]$} radially decreasing if $f(x_1)\geq f(x_2)$ whenever $|x_1|\leq |x_2|$. We observe the following three auxiliary results.
\begin{lemma}\label{lem:1e}
Let  $r>0$ and let $f\geq 0$ be constant on $B(0,r)$ and radially decreasing. Then,
$$
f* 1_{B(0,r)} \geq |B(0,r/2)| f \,.
$$
\end{lemma}
\begin{proof}
We have $f*1_{B(0,r)}(x)=\int_{B(x,r)}f(y)dy$. If $|x|<r$, then
$$f*1_{B(0,r)}(x)\ge f(0)|B(0,r)\cap B(x,r)|\ge f(0)|B(0,r/2)|=f(x)|B(0,r/2)|,
$$
where $|B(0,r/2)|$ denotes the volume of $B(0,r/2)$.
If $|x|\ge r$, then
$$f*1_{B(0,r)}(x)\ge f(x)|B(0,|x|)\cap B(x,r)|\ge f(x)|B(0,r/2)|.$$
\end{proof}
\begin{lemma}\label{lem:2e}
Let $0\le k\le K$ be radially decreasing and fix $r>0$. Let $c_1
=\int_{\Rd}k
(x)dx$, 
$c_2
=K(r,0,\ldots,0)|B(0,r/2)| $. Let $c_3=1$ if $c_2=0$ or $\infty$, and let $c_3=1+c_1/c_2$ otherwise.
Then,
\begin{align*}
\sup_{x\in\Rd} \int_{\Rd} |U(z)| k
(x-z)\,dzdu
\leq  
c_3
\sup_{x\in\Rd} \int_{B(x,r)} |U(z)| K(x-z)\,dz\,.
\end{align*}
\end{lemma}
\begin{proof}
Define
$f
(x)=k(x)
\land k
(r,0,\ldots,0)$.
Assume first that $0<c_2
<\infty$. By Lemma \ref{lem:1e},
\begin{align*}
k &\leq 1_{B(0,r)}K +f\le  1_{B(0,r)}K + 1_{B(0,r)}*f/|B(0,r/2)|\\
&\le  1_{B(0,r)}K +
( 1_{B(0,r)}K)*f/c_2.
\end{align*}
The inequality in the statement of the lemma follows from this, 
because
\begin{align*}
|U|* k \leq |U|* (1_{B(0,r)}K ) * \left( \delta_0 +f/c_2 \right)
 \leq 
 \sup \left[|U|* (1_{B(0,r)}K)\right]\left(1+c_1/c_2\right).
\end{align*}
If $c_2=0$, then the lemma follows immediately with $c_3=1$, since then $k\leq K= K 1_{B(0,r)}$. If $c_2=\infty$, then we have $K=\infty$ on $B(0,r)$ and the lemma is trivially true with $c_3=1$.
\end{proof}
\begin{proposition}\label{prop:nice}
If $c_0= c_0(d)=\frac{\Gamma(d/2-1)}{ 4\pi^{d/2}}$, $c>0$, $\tau>0$, $r>0$ and $U\colon\Rd\to\RR$, then
\begin{align*}
\sup_{s\in \RR,x\in\Rd} \int_s^{s+\tau} \int_{\Rd} g_c(s,x,u,z)|U(z)|\,dzdu
\leq \left( c\,c_0+ \frac{\tau}{r^2 |B(0,1/2)|} \right) I_{r}(U)\new{\,.}
\end{align*}
\end{proposition}
\begin{proof}
We let $k(x)=\int_0^{\tau}g_c(0,0,u,x)du$, $K(x)=\int_0^{\infty}g_c(0,0,u,x)du =  c\,c_0 |x|^{2-d}$  in Lemma \ref{lem:2e} and observe that
$c_1= \tau$ and $c_2= |B(0,r/2)|K(r,0,\ldots,0)=|B(0,1/2)|c\,\,c_0\, r^2$.
\end{proof}

\begin{proof}[Proof of \eqref{ecK}]
$\tau=h$ and $r=\sqrt{h}$ in Proposition \ref{prop:nice} yield the constant \eqref{ecK} in \eqref{eq:psup}. 
\end{proof}
By Theorem~\ref{thm:4P} \new{(replacing $\vee$ by $+$)} and Proposition \ref{prop:nice} with $\tau=t-s$ and $r=\sqrt{h}$, for 
$q:\Rd \to [0,\infty]$ we have
\begin{align}
\int_s^t \int_{\Rd} &g_b (s,x,u,z) q(z) g_a (u,z,t,y)\,dzdu \nonumber \\
&\leq  
\new{I_{\sqrt{h}}(q)}  
M \Big[b\,c_0+\frac{2(t-s)}{h |B(0,1/2)|} \Big]g_a (s,x,t,y)\,
\new{.} \label{explKato}
\end{align}
We are in a position to summarize part of our discussion as given by Theorem~\ref{thm:new}.
\begin{proof}[Proof of Theorem~\ref{thm:new}]
We consider \new{a} (Borel measurable) transition density $p$ on space-time, where the {\it space} is
$X=\Rd$ with 
the Lebesgue measure $dz$.
Let $0<a<b$ and $\sp(s,x,t,y)=e^{\lambda(t-s)}g_a(s,x,t,y)$. In view of \eqref{ineq:gbga} we may take $C=(b/a)^{d/2}\Lambda$ in \eqref{eq:pCtp}. If $q\in \pN(g_b,g_a,(b/a)^{d/2}, \eta,Q)$,
then $q\in\pN(p,\sp,C,\Lambda\eta,\Lambda Q)$, and the assertion follows from Theorem~\ref{thm:1}.
The inequality \eqref{explKato} means that
\new{a} time-independent $q$ is in $\in\pN(g_b,g_a,(b/a)^{d/2},\eta,Q)$ with 
\begin{align*}
\eta= b\, c_0 M I_{\sqrt{h}}(q)  
\,,\qquad
Q(s,t)=(t-s)\frac{2M}{h |B(0,1/2)|} I_{\sqrt{h}}(q)\,,  
\end{align*}
provided these are finite.
\end{proof}

We also observe the following characterization of $\pN(g_b,g_a,(b/a)^{d/2},\eta,Q)$.
\begin{corollary}\label{charN}
If $d\geq 3$, $0<a<b$ and
$q\colon\Rd\to[0,\infty]$, then $q\in \pN(g_b,g_a,(b/a)^{d/2},\eta,Q)$ for some $\eta$ and $Q$ if and only if $I_{\sqrt{h}}(q)<\infty$  
for some (hence for all) $h>0$.
\end{corollary}
\begin{proof}
If $q\in\pN(g_b,g_a,(b/a)^{d/2},\eta,Q)$, then integrating \eqref{def:coeNs} in $y$, we obtain
$$
\sup_{x\in\Rd} \int_0^h \int_{\Rd} g_b (0,x,u,z) q(z)\,dzdu 
\leq \eta + Q(0,h) \,.
$$
By the lower bound of \eqref{eq:psup} we obtain
$$
I_{\sqrt{h}}(q) 
\le C_0^{-1}(d,b)\left(\eta+Q(0,h)\right)<\infty.
$$
The converse implication follows from \eqref{explKato}, which also shows that $Q$ may be taken linear.
\end{proof}

We now pass to
{\it \new{a} parabolic Kato condition}.
{For $c>0$, $h>0$ and $V: \RR\times \Rd\to \RR$ we denote
$$
N_h^c (V) \!= \sup_{s, x} \int_s^{s+h} \!\! \int_{\Rd} \!\! g_c(s,x,u,z)|V(u,z)|dzdu
+ \sup_{t, y}  \int_{t-h}^t\!  \int_{\Rd}\!\!  g_c(u,z,t,y)|V(u,z)|dzdu.
$$
We say that $V$ is of 
{\it parabolic} Kato class if $\lim_{h \to 0^+} N_h^c (V)=0$ for every $c>0$. 
Considering $V(s,x)=U(x)$ for $s\in \RR$, $x\in \Rd$, we may 
regard the parabolic Kato class as wider than 
the (time-independent) Kato class.}
We note that $N_h^c (V)$ is non-decreasing in $h$. 
\new{Obviously,}
$N_{h_1+h_2}^c (V) \leq N_{h_1}^c (V) + N_{h_2}^c (V)$, hence
\begin{align}\label{step:2}
N_{t-s}^c (V)\leq N_h^c (V) + N_h^c (V) (t-s)/h, \quad h>0\,.
\end{align} 
To focus on nonnegative Schr\"odinger perturbations (in this connection see Remark~\ref{sper}), we consider, as before\new{, a} function $q\ge 0$ on space-time.
If $0<a<b<\infty$, then by \new{the} 4G Theorem,  there 
is an explicit constant $M'$ depending only on $d$ and $b/a$, such that for all $s<t$ and $x,y \in \Rd$,
\begin{align}\label{step:1}
\int_s^t \int_{\Rd} g_b (s,x,u,z) q(u,z) g_a (u,z,t,y)\,dzdu \leq M' N_{t-s}^c (q)\; g_a (s,x,t,y)\,,
\end{align}
where $c=(b-a)\land a$. In fact, we may take $M'=\left(\frac{b-a}{a} \vee \frac{a}{b-a}\right)^{d/2}M$, where $M$ is the constant in Theorem~\ref{thm:4P}.
If $q\ge 0$ belongs to the parabolic Kato class and $0<a<b<\infty$, 
then by (\ref{step:2}) and (\ref{step:1}) we have 
\begin{equation}\label{pq}
q\in \pN(g_b,g_a,(b/a)^{d/2},\eta,Q),
\end{equation}
with 
$Q(s,t)= \beta (t-s)$
and $\beta = \eta /h$, 
provided 
$h>0$ and $\eta>0$ are such that $N_h^{(b-a)\land a} (q)\leq \eta/M'$. Then, in view toward applying Theorem~\ref{thm:1}, we are free to choose arbitrarily small $\eta>0$ in \eqref{pq}, at the expense of having large $\beta$.

\begin{remark}\label{pdK}
The condition $q\in \pN(g_b,g_a,(b/a)^{d/2}, \eta, Q)$ invites a trade-off between $\eta$ and $Q$.
In particular, it follows from the discussion of \eqref{explKato} and \eqref{pq} that
arbitrarily small $\eta>0$ and a linear, possibly large, but explicit $Q$ exist
if $q$ is in the Kato class or the parabolic Kato class.
\end{remark}

\begin{remark}\label{rK}
The (time-independent) Kato class was first used to perturb the Laplace operator by Aizenman and Simon \cite{MR644024}, and was characterized 
as smallness with respect to the Laplacian on $L^1(\Rd)$. The parabolic Kato class was proposed for the Gaussian kernel by Zhang in \cite{MR1488344}. It was then generalized and used by Liskevich and Semenov
\cite{MR1783642}, Liskevich, Vogt and Voigt
\cite{MR2253015} and Gulisashvili and van Casteren \cite{MR2253111}. The condition is related to Miyadera perturbations of the semigroup of the Laplacian on $L^1(\Rd)$, see Schnaubelt and Voigt \cite{MR1687500}.
\new{The time-independent} Kato class is wider than $L^{p}(\Rd)$ if $p>d/2$
\cite{MR644024}, \cite[Chapter 3., Example 2]{MR1329992}.
Nevertheless, the latter space is quite 
natural for perturbing Gaussian kernels, see Aronson \cite{MR0435594}, Dziuba{\'n}ski and Zienkiewicz \cite{MR2164260}, \cite[Remark~1.1(b)]{MR1978999}. Another Kato-type condition was introduced by Zhang in \cite{MR1978999} to obtain strict comparability of $g$ and $\tilde{g}$. 
As noted in \cite[Remark~1.1(c)]{MR1978999}, the condition may be formulated in terms of {\it Brownian bridges}.
This point of view was later developed in \cite{MR2457489} (under the name of the relative Kato condition) and elaborated
in \cite{MR3000465} to 
\begin{equation}\label{eq:kk_2}
\int_s^t\int_X \frac{p(s,x,u,z)p(u,z,t,y)}{p(s,x,t,y)}q(u,z)dzdu\leq [\eta +Q(s,t)],
\end{equation} 
where $s<t$, $x,y \in X$, $\eta<\infty$, and $0\le Q(s,u)+Q(u,t)\le Q(s,t)$, cf. \eqref{eq:kk_1}.
Condition \eqref{eq:kk_2} 
indicates why we mention bridges here (see \cite{MR2457489} for details).
The Kato condition for bridges gives better upper bounds and seems more intrinsic to Schr\"odinger perturbations than the parabolic Kato condition, but the former may be cumbersome to verify in 
concrete situations. 
For the classical Gaussian kernel, \eqref{eq:kk_2} 
is stronger than the corresponding parabolic Kato condition with a fixed $c$ (see \cite[Lemma~9]{MR2457489} for a more general result), 
and it is rather difficult to explicitly characterize \cite[Remark~1.2(a,b)]{MR1978999}.
 This is due to the relatively large values of the 
integrand in \eqref{eq:kk_2} for $(u,z)$ on the interval connecting $(s,x)$ and $(t,y)$.
If $p$ satisfies \new{the} 3G inequality, $p(s,x,t,y)=p(s,y,t,x)$ and $p$ is a probability transition density,
then the parabolic Kato class and 
the Kato class for bridges 
\new{coincide}. This is the case for the transition density of the fractional Laplacian $\Delta^{\alpha/2}$ with $\alpha\in (0,2)$ \cite[Corollary~11]{MR2457489}, and the proof of this fact is similar to our application of 4G in \eqref{step:1}.
We emphasize
that each transition density $p$ determines its specific Kato classes (either parabolic or for bridges), and detailed analysis is required to 
manage particularly singular $q$. 
\end{remark}

Let $d\geq 3$, $z_1 \in \Rd$, $|z_1|=1$, and $B(nz_1,1/n)\subset \Rd$ be the ball with radius $1/n$ and center $nz_1$, $n=1,2,\ldots$. We define
$$U(z)=\sum_{n=2}^{\infty} n|z-nz_1|^{-1} \chiI_{B(nz_1,1/n)}(z),\quad z\in \Rd.$$
(A similar example appears in \cite[Appendix 1]{MR644024}.)  If $\delta>0$ and $n\geq 1/\delta$, then
\begin{align*}
I_{\delta}(U) \geq \int_{B(nz_1, 1/n)} n |z-nz_1|^{-d+1}\,dz =\int_{B(0,1)}|z|^{-d+1}\,dz\,.
\end{align*}
Therefore $\varepsilon U$ does not belong to the parabolic Kato class for any   
$\varepsilon>0$.
On the other hand, $I_1(U)<\infty$, and by \eqref{explKato}, {$\varepsilon U \in \pN(g_b,g_a,(b/a)^{d/2},\eta,Q)$}   
with a {\it linear} $Q$ and small $\eta$, provided $\varepsilon$ is sufficiently small, cf. Corollary~\ref{charN} and Theorem~\ref{thm:1}.
A similar effect may be obtained for the original $U$ if we instead make $b$ smaller while keeping $b/a$ constant. 
Since our constant in \eqref{explKato} is explicit, so are the resulting upper 
bounds for $\tilde{g}_b$.
Similar conclusions obtain in the generality of Theorem~\ref{thm:new}, and applications to eventual estimates of transition densities are presented below.

To summarize this discussion of examples of $q$ manageable by our methods, we let $q(u,z):=U(z)+
\varphi(u,z)$, where $\varphi \ge 0$, $U$ is as above, and $f(u)=\sup_{z\in \Rd} \varphi(u,z)$ is finite and increases to infinity as $u\to \infty$. Such $q$ requires $\eta>0$ to control $U$, and a {\it superlinear} 
$Q$ to majorize $\int_s^t f(u)\,du$ (cf. the discussion after the proof of Theorem \ref{thm:1}).

\vspace{10pt}

\begin{example}\label{exLap}
\begin{rm}
If $p=g_b$ and $p^*=g_a$, then we 
take $\Lambda=1$, $C=(b/a)^{d/2}$ and $\lambda=0$  in  \eqref{eq:Gub} and, consequently, in  
\eqref{eq:geKa}.
For clarity, $a$, the coefficient in the exponent of the Gaussian majorant, may be arbitrarily close to $b$, at the expense of the factor before the majorant in \eqref{eq:geKa},
and we require that  $q\in \pN(g_b,g_a,C, \eta,Q)$ with 
$\eta\in [0, 1)$,
which is
satisfied, e.g., for $q$ in 
the Kato class.
We thus
recover the best results known in this setting \cite[proof of Theorem A]{MR1488344} with explicit control of constants.\\
To relate our results to second order differential operators, we let $C_c^{\infty}(\RR\times\Rd)$ denote the smooth compactly supported 
functions on space-time, and we recall that for all $s\in \RR$, $x\in\Rd$ and $\phi \in C_c^{\infty}(\RR\times\Rd)$, 
\begin{align}\label{fsb}
\int_{-\infty}^{\infty} \int_{\Rd} p(s,x,u,z)
\left[\frac{\partial \phi(u,z)}{\partial u}+\frac1b\Delta \phi(u,z)\right]\,dzdu=-\phi(s,x)\,.
\end{align}
The identity may, for instance, be obtained from integration by parts or by using Fourier transform in the space variable. 
By a general result, \cite[Lemma~4]{MR3000465},  
the perturbed transition density $\tp$ corresponds to the Schr{\"o}dinger-type operator 
$\frac1b\Delta + q$ in the same way,
\begin{align*}
\int_s^{\infty} \int_{\Rd} \tp(s,x,u,z)
\left[\frac{\partial \phi(u,z)}{\partial u}+\frac1b\Delta \phi(u,z)
+q(u,z)\phi(u,z)
\right]\,dzdu=-\phi(s,x)\,,
\end{align*}
provided $q\in \pN(g_b,g_a,C, \eta,Q)$ with 
$\eta\in [0, 1)$, as above.
\end{rm}
\end{example}

\begin{example}
\begin{rm}
Let $p$ be the transition density of the one-dimensional Brownian motion with constant unit drift, 
\begin{align*}
p(s,x,t,y) = g_1 (s,x-s,t,y-t)\,,\qquad s<t,\;x,y\in \RR\,. 
\end{align*}
There are no constants $c_1, c_2$ such that $p(s,x,t,y)\leq c_1 g_{c_2}(s,x,t,y)$ for all $s<t$ and $x,y\in\RR$ (cf. Zhang \cite[Remark~1.3]{MR1457736}). On the other hand, for each $b\in (0,1)$ we have
\begin{align*}
p(s,x,t,y)\leq b^{-1/2}e^{\frac{b}{4(1-b)}(t-s)} g_b(s,x,t,y)\,, \qquad s<t,\;x,y\in \RR\,. 
\end{align*}
This shows why we may need $\lambda\neq 0$ in \eqref{eq:Gub}, see also Norris \cite{MR1482931}. 
\end{rm}
\end{example}

\begin{example}\label{exD}
\begin{rm}
Let 
\new{$f:\RR\times\Rd\mapsto \RR$ be a function of $(s,x)\in \RR\times\Rd$} and 
\begin{align*}
L f = \sum_{i,j=1}^n a_{ij}(s,x)\frac{\partial^2f}{\partial x_i\partial x_j}+
\sum_{i=1}^n b_i(s,x)\frac{\partial f}{\partial x_i}\,,
\end{align*}
be 
a uniformly elliptic operator, with bounded uniformly H\"older continuous coefficients $b_i$ and $a_{ij}=a_{ji}$,  see Dynkin \cite[Chapter 2, 1.1.A and 1.1.B]{MR1883198}  for detailed definitions, and \cite[Chapter~1]{MR0181836} for \new{a} wider perspective.
Consider the fundamental solution $p(s,x,t,y)$ in the sense of \cite[Theorem~1.1]{MR1883198} for the following   
\new{parabolic differential} operator
\begin{equation}\label{cpo}
\frac{\partial f}{\partial s} + Lf\,.
\end{equation}
By the results of \cite[ Chapter 2]{MR1883198}, in particular  Theorem 1.1, 1.3.1 and 1.3.3, $p$ satisfies our assumptions, including \eqref{eq:Gub} with $-\infty<t_1\le s<t \le t_2<\infty$.
Therefore the Schr\"odinger perturbation $\tilde p$ of  $p$ 
satisfies (\ref{eq:geKa}) 
in 
the (finite) time 
horizon $[t_1,t_2]$, if $q\in \pN(g_b,g_a,(b/a)^{d/2}, \eta,Q)$ and 
$\eta\in [0,1/\Lambda)$,
as explained 
 after the proof of Theorem~\ref{thm:1}.
We thus recover recent results of \cite[Theorem~3.10 and 3.12]{MR2253015} (see also Remark~\ref{rfs} below).
We now explain how 
$p$ and $\tilde p$ are related to 
parabolic operators.
For all $s\in \RR$, $x\in\Rd$ and $\phi \in C_c^{\infty}(\RR\times\Rd)$, we have
\begin{align}\label{rfund}
\int_s^{\infty} \int_{\Rd} p(s,x,u,z)
\left[
\frac{\partial \phi(u,z)}
{\partial u}
+L\phi(u,z)\right]\,dzdu=-\phi(s,x)\,.
\end{align}
In fact, the 
identity holds if the function $\phi(s,x)$ is bounded, supported in a finite time interval and 
uniformly H\"older continuous in $x$, and if the same 
is true for
its first derivative in time and 
all its
derivatives up to the second order in space.
Indeed, if we let
\begin{align*}
h(s,x)=  \phi(s,x) + \int_s^{\infty}\int_{\Rd} p(s,x,u,z)\left[ \frac{\partial \phi (u,z)}{\partial u}+ L\phi(u,z)\right]\,dzdu\,,
\end{align*}
then by \cite[Chapter 2, 1.3.3 ]{MR1883198} we have $h\equiv 0$, which verifies \eqref{rfund}.
By \eqref{rfund}
and \cite[Lemma~4]{MR3000465}, 
the perturbed transition density $\tp$ corresponds to the Schr{\"o}dinger-type operator 
$L + q$ in a similar way: for all $s\in \RR$, $x\in\Rd$ and $\phi \in C_c^{\infty}(\RR\times\Rd)$,
\begin{align*}
\int_s^{\infty} \int_{\Rd} \tp (s,x,u,z) \left[ \frac{\partial \phi(u,z)}{\partial u} + L\phi(u,z)+ q(u,z)\phi(u,z)\right] \,dzdu=-\phi(s,x)\,.
\end{align*}
\end{rm}
\end{example}

\begin{remark}\label{rfs}
\rm In this paper by a fundamental solution we mean the negative of an integral inverse of a given operator acting
on space-time (other authors also use the terms heat kernel and Green function).
More specifically, our $p$ and $\tilde p$ 
may be considered 
post-inverses of the respective differential operators
acting on $C^\infty_c(\RR\times \Rd)$, cf. Example~\ref{exLap} above and
\cite{MR1978999}, \cite[p. 13]{MR2320609}.
In the literature on 
partial differential equations it is 
also common to consider 
the 
pre-inverses, which necessitate sufficient differentiability of the applications of $p$, $\tp$ to test functions  \cite[Theorem I.5.9]{MR0181836}, \cite[(1.12)]{MR1883198}. 
Still differently, if 
the operator $L$ is in the divergence form, a common
notion is that of the weak fundamental solution, related to integration by parts, see
Aronson \cite[(2), (8), Section 6 and 7]{MR0435594}, 
Cho, Kim and Park
\cite{MR2853532} and Liskevich and Semenov
\cite{MR1783642}.
It is also 
customary to 
study the operator $\partial f/\partial s - Lf$, which 
is related to \eqref{cpo}  
by time reversal $s\mapsto -s$ \cite{MR1883198,MR2253111}, \cite[(7.3)]{MR0435594}.
The setting of \eqref{cpo} and \eqref{rfund}
is most appropriate from the probabilistic point of view: it agrees with time precedence and notation for (measurable) transition densities of Markov processes, which may be conveniently considered as integral operators on space-time.
\end{remark}
\begin{remark}
\rm
As we already mentioned, Zhang \cite{MR1978999} gives sharp estimates for perturbations of $p=g$. Sharp Gaussian estimates (corresponding to $p^*=p$) are generally not available by our methods if (plain) Kato condition and 4G should be used to estimate $p_1$.
Accordingly,
\cite{MR1978999} assumes an integral condition related to the Brownian bridge, 
to bound $p_1$ by $p$. 
As we explained 
in Remark~\ref{rK},  the Kato condition for bridges is more restrictive than the parabolic Kato condition (a straightforward general approach using bridges is 
given in \cite{MR3000465}).
\end{remark}

We now comment on the uniqueness of $\tilde p$. 
Trivially, $\tilde p$ is unique because it is given by \eqref{def:tp}, rather than implicitly.   
However, in the literature of the subject a departure point for defining $\tilde p$ is usually one of the following Duhamel's (perturbation/resolvent) formulas,
\begin{equation}\label{Df}
\tilde p(s,x,t,y)=p(s,x,t,y)+\int_s^t \int_\Rd p(s,x,u,z)q(u,z)\tp(u,z,t,y)dudz,
\end{equation}
\begin{equation}\label{Df1}
\tilde p(s,x,t,y)=p(s,x,t,y)+\int_s^t \int_\Rd \tp(s,x,u,z)q(u,z)p(u,z,t,y)dudz.
\end{equation}
In short, $\tp=p+pq\tp$ {\it or} $\tp=p+\tp qp$, depending on whether we consider $p$ and $\tp$ as pre- {\it or} post-inverses, respectively, of the corresponding differential operators (see \cite{MR3000465,2011-BHJ} for notation related to integral kernels).
Clearly, \eqref{def:tp} yields \eqref{Df} and \eqref{Df1}.
Conversely, iterating \eqref{Df} or \eqref{Df1} we 
get \eqref{def:tp}, and uniqueness, if 
$(pq)^n\tp$ or $\tp (qp)^n$ 
converge to zero as $n\to \infty$. So is the case with $(pq)^n\tp$ under the assumptions of Theorem~\ref{thm:1}, {\it provided} 
$\tilde p$ is locally in time majorized by a constant multiple of $p^*$, because then $(pq)^np^*\to 0$.
We refer to \cite[Theorem~2]{MR2457489} and \cite[Theorem~1.16]{MR2253015} for further discussion of the perturbation formula and uniqueness. 
We also note that in the setting of {\it bridges}, there is a {natural} probabilistic Feynman-Kac type formula for $\tp$ \cite[Section~6]{MR2457489}, which readily yields \eqref{def:tp} and uniqueness. 

Here is a general
argument leading to \eqref{Df1}. We consider $-p$ and $-\tp$ as integral operators on space-time and 
post-inverses of operators $\mathcal{L}$ and $\mathcal{L}+q$, respectively, which in turn act on the same given set of functions. We have  
$\tp (\mathcal{L}\phi+q\phi)=-\phi=p\mathcal{L}\phi$, hence $\tp \psi=p\psi+\tilde p q p\psi$, 
where $\psi=\mathcal{L}\phi$. If the range of $\mathcal{L}$ uniquely determines measures, then we 
obtain $\tp=p+\tp q p$ as integral kernels.
This is the case, e.g., in the context of Example~\ref{exD}.
We finally note that some majorization of $\tp$ is needed for uniqueness. For instance, both $p=g_b$ and $p(s,x,u,z)=g_b(s,x,u,z)+2du+b |z|^2$ satisfy \eqref{fsb}.

\section{Miscellanea}\label{sec:Misc}

Earlier work by Jakubowski \cite{MR2507445} and coauthors \cite{MR3000465} in slightly different settings does not require continuity assumptions on $Q$. Namely
we call $Q \colon \RR \times \RR \to [0,\infty)$ {\it superadditive}, if
\begin{equation}\label{ineq:Qp}
Q(s,u) + Q(u,t) \leq Q(s,t)\,, \qquad \mbox{for}\quad s< u< t \,.
\end{equation}
For convenience we define $Q(s,t)=0$ if $s\geq t$.
We see that $t\mapsto Q(s,t)$ is non-decreasing, $s\mapsto Q(s,t)$ is non-increasing, 
and the following limit exists,
\begin{equation}\label{eq:defQ-p}
Q^-(s,t)=\lim_{h\to 0^+}Q (s+h,t-h)\,.
\end{equation}
For instance, if $\mu$ is a Radon measure on $\RR$ and $Q(s,t)=\mu(\{u\in \RR: s\le u <t\})$, then $Q^-(s,t)=\mu(\{u\in \RR: s< u <t\})$.

Clearly, $0\leq Q^-(s,t)\leq Q(s,t)$ and $Q^-(s,u)+Q^-(u,t)\leq Q^-(s,t)$ if $s\leq u \leq t$. We have $Q^{--}=Q^-$. In fact, $Q^-(u,v)\to Q^-(s,t)$ as $u\to s^+$, $v\to t^-$, because if $0<h< u-s< k$ and $h< t-v < k$, then 
$Q(s+k,t-k)\leq Q^-(u,v)\leq Q(s+h,t-h)$.  
In particular, $s\mapsto Q^-(s,t)$ is right-continuous and $t\mapsto Q^-(s,t)$ is left-continuous. 
We thus obtain the following result.
\begin{corollary}
$Q^-(s,t)$ is 
regular superadditive.
\end{corollary}
We note that continuous superadditive functions are used in the theory of rough paths by Lyons \cite{MR2314753}. There are further similarities due to the role of iterated integrals in time here and in \cite{MR2314753}, and many differences related to the fact that we require absolute integrability (but see \cite{MR2875353}) and also integrate/average in space (see \eqref{def:coeNs} and \eqref{defpn}).
We also note that methods similar to ours allow to handle gradient perturbations, which will be discussed in a forthcoming paper (see also \cite{MR2876511,MR2875353}).

The next
result shows that if $Q$ is (plain) superadditive, then
the 
factor $\eta+Q(s,t)$ in (\ref{def:coeNs}) may be replaced with $C\eta+CQ^-(s,t)$, where $C$ comes from \eqref{eq:pCtp}, and $CQ^-$ is regular superadditive.
Thus, we may assure {\it regular} superadditivity at the expense of increasing $\eta$ and $Q$.
\begin{lemma} \label{lem:3}
Assume that $p$ and $\sp$ are transition densities, function $q\geq 0$ is defined (and measurable) on space-time, $\eta\geq 0$, $C\geq 1$, $Q$ is superadditive, and {\rm (\ref{eq:pCtp})} and  {\rm (\ref{def:coeNs})} hold.
Then for all $s<t$ and $x,y\in X$, we have
\begin{equation*}
\int_s^t \int_{X} p(s,x,u,z)q(u,z) \sp(u,z,t,y)\,dz\,du \leq C\big[ \eta + Q^-(s,t)\big] \sp(s,x,t,y)\,.
\end{equation*}
\end{lemma}

\begin{proof}
\new{For $x,y\in X$, $s<t$, $0<h<(t-s)/2$, taking $\gamma=\eta +Q(s+h,t-h)$, by Chapman-Kolmogorov and \eqref{def:coeNs}, we get
\begin{align*}
&\int_{s+h}^{t-h} \int_X p(s,x,u,z)q(u,z)\sp(u,z,t,y)\, dzdu 
=\int_X \int_X p(s,x,s+h,v) \\ 
&\int_{s+h}^{t-h} \int_X p(s+h,v,u,z) q(u,z)
\sp(u,z,t-h,w)dzdu\, \sp(t-h,w,t,y)\, dw dv\\
&\le \gamma \int_X \int_Xp(s,x,s+h,v) \sp(s+h,v,t-h,w) \sp(t-h,w,t,y)\, dw dv\,.
\end{align*}
Again by Chapman-Kolmogorov, the above equals
\begin{align}
\big[ \eta +Q(s+h,t-h) \big] \int_X p(s,x,s+h,v) \sp(s+h,v,t,y)\,dv \,,\label{eq:Qpsp}
\end{align}
which
leads to the upper bound $C \big[ \eta +Q(s+h,t-h) \big]\sp(s,x,t,y)$, by \eqref{eq:pCtp} and Chapman-Kolmogorov. We then let $h\to 0^+$, and use \eqref{eq:defQ-p} and the monotone convergence theorem, ending the proof.
}
\end{proof}

If $p^*$ is a time-changed $p$, then
we can do even better.
\begin{lemma} \label{lem:QQ}
Under the assumptions of Lemma~\ref{lem:3} we have 
\begin{equation*}
\int_s^t \int_{X} p(s,x,u,z)q(u,z) \sp(u,z,t,y)\,dz\,du \leq \big[ \eta + Q^-(s,t)\big] \sp(s,x,t,y)\,,
\end{equation*}
if  
\new{$t\mapsto p(s,x,t,y)$, $t\in(s,\infty)$,} is continuous, $p$ is time-homogeneous: $p(s,x,t,y)=p(s+r,x,t+r,y)$ for $r\in \RR$,  and $\sp(s,x,t,y)=p(as,x,at,y)$ for some $a>0$.
\end{lemma}
\begin{proof}
Picking up \eqref{eq:Qpsp}, for $s<t$, $x,y\in\Rd$, we have
\new{\begin{align*}
&\limsup_{h\to 0^+} \int_{X} p (s,x,s+h,v) \sp (s+h,v, t, y)\,dv \\ 
&= \limsup_{h\to 0^+} \int_{X} p (s,x,s+h,v) p(s+h,v,s+h+a(t-s-h),y) \,dv\\
&= \limsup_{h\to 0^+} \sp(s,x,t-h+h/a,y) = \sp(s,x,t,y)\,. 
\end{align*}}
\end{proof}
Lemma~\ref{lem:QQ} applies to the Gaussian density, if $p=g_b$ and $\sp=g_a$, where $0<a<b$. Indeed, $g_b(s,x,t,y)=  g_a(as/b,x,at/b,y)$, see also \eqref{ineq:gbga}.

\subsection*{Acknowledgements}
\label{ackref}
 We thank Tomasz Jakubowski, Panki Kim,  George Papanicolaou and Lenya Ryzhik  for discussions and references. {Krzysztof Bogdan gratefully acknowledges the hospitality of the Departments of Mathematics and Statistics at Stanford University, where the paper was written in part.} We thank Tomasz Grzywny for helpful suggestions leading to \eqref{explKato} and Lemma~\ref{lem:2e}. \new{We thank the referees for helpful suggestions.}


\begin{thebibliography}{99}







\normalsize
\baselineskip=17pt


\bibitem{MR644024}
{M.~Aizenman \and B.~Simon},
\newblock \emph{Brownian motion and {H}arnack inequality for {S}chr\"odinger
  operators},
\newblock {Comm. Pure Appl. Math.} 35(2):209--273, 1982.

\bibitem{MR0435594}
{D.~G. Aronson},
\newblock \emph{Non-negative solutions of linear parabolic equations},
\newblock {Ann. Scuola Norm. Sup. Pisa (3)} 22:607--694, 1968.

\bibitem{MR2457489}
{K.~Bogdan, W.~Hansen \and T.~Jakubowski},
\newblock \emph{Time-dependent {S}chr\"odinger perturbations of transition densities},
\newblock {Studia Math.} 189(3):235--254, 2008.


\bibitem{2011-BHJ}
{K.~Bogdan, W.~Hansen \and T.~Jakubowski},
\newblock \emph{Localization and {Schr�dinger} perturbations of kernels},
\newblock {Potential Anal.} pages 1--16, 2012.
\newblock http://dx.doi.org/10.1007/s11118-012-9320-y.

\bibitem{MR2283957}
{K.~Bogdan \and T.~Jakubowski},
\newblock \emph{Estimates of heat kernel of fractional {L}aplacian perturbed by
  gradient operators},
\newblock {Comm. Math. Phys.} 271(1):179--198, 2007.

\bibitem{MR3000465}
{K.~Bogdan, T.~Jakubowski \and S.~Sydor},
\newblock \emph{Estimates of perturbation series for kernels},
\newblock {J. Evol. Equ.} 12(4):973--984, 2012.

\bibitem{MR2853532}
{S.~Cho, P.~Kim \and H.~Park},
\newblock \emph{Two-sided estimates on {D}irichlet heat kernels for time-dependent parabolic operators with singular drifts in {$C^{1,\alpha}$}-domains},
\newblock {J. Differential Equations} 252(2):1101--1145, 2012.

\bibitem{MR1329992}
{K.~L. Chung \and Z.~X. Zhao},
\newblock \emph{From {B}rownian motion to {S}chr\"odinger's equation}, volume
  312 of {Grundlehren der Mathematischen Wissenschaften [Fundamental
  Principles of Mathematical Sciences]}.
\newblock Springer-Verlag, Berlin, 1995.

\bibitem{MR936811}
{M.~Cranston, E.~Fabes \and Z.~Zhao},
\newblock \emph{Conditional gauge and potential theory for the {S}chr\"odinger
  operator},
\newblock {Trans. Amer. Math. Soc.} 307(1):171--194, 1988.

\bibitem{MR1883198}
{E.~B. Dynkin},
\newblock \emph{Diffusions, superdiffusions and partial differential equations},
  volume~50 of {American Mathematical Society Colloquium Publications}.
\newblock American Mathematical Society, Providence, RI, 2002.

\bibitem{MR2164260}
{J.~Dziuba{\'n}ski \and J.~Zienkiewicz},
\newblock \emph{Hardy spaces {$H^1$} for {S}chr\"odinger operators with compactly
  supported potentials},
\newblock {Ann. Mat. Pura Appl. (4)} 184(3):315--326, 2005.

\bibitem{MR0181836}
{A.~Friedman},
\newblock \emph{Partial differential equations of parabolic type}.
\newblock Prentice-Hall Inc., Englewood Cliffs, N.J., 1964.

\bibitem{MR2253111}
{A.~Gulisashvili \and J.~A. van Casteren},
\newblock \emph{Non-autonomous {K}ato classes and {F}eynman-{K}ac propagators}.
\newblock World Scientific Publishing Co. Pte. Ltd., Hackensack, NJ, 2006.

\bibitem{MR2207878}
{W.~Hansen},
\newblock \emph{Global comparison of perturbed {G}reen functions},
\newblock {Math. Ann.} 334(3):643--678, 2006.

\bibitem{MR2507445}
{T.~Jakubowski},
\newblock \emph{On combinatorics of {S}chr\"odinger perturbations},
\newblock {Potential Anal.} 31(1):45--55, 2009.

\bibitem{MR2875353}
{T.~Jakubowski},
\newblock \emph{Fractional {L}aplacian with singular drift},
\newblock {Studia Math.} 207(3):257--273, 2011.

\bibitem{MR2643799}
{T.~Jakubowski \and K.~Szczypkowski},
\newblock \emph{Time-dependent gradient perturbations of fractional {L}aplacian},
\newblock {J. Evol. Equ.} 10(2):319--339, 2010.

\bibitem{MR2876511}
{T.~Jakubowski \and K.~Szczypkowski},
\newblock \emph{Estimates of gradient perturbation series},
\newblock {J. Math. Anal. Appl.} 389(1):452--460, 2012.

\bibitem{MR1783642}
{V.~Liskevich \and Y.~Semenov},
\newblock \emph{Estimates for fundamental solutions of second-order parabolic
  equations},
\newblock {J. London Math. Soc. (2)} 62(2):521--543, 2000.

\bibitem{MR2253015}
{V.~Liskevich, H.~Vogt and J.~Voigt},
\newblock \emph{Gaussian bounds for propagators perturbed by potentials},
\newblock {J. Funct. Anal.} 238(1):245--277, 2006.

\bibitem{MR2314753}
{T.~J. Lyons, M.~Caruana \and T.~L{\'e}vy},
\newblock \emph{Differential equations driven by rough paths}, volume 1908 of
  {Lecture Notes in Mathematics}.
\newblock Springer, Berlin, 2007.
\newblock Lectures from the 34th Summer School on Probability Theory,
  Saint-Flour, July 6--24, 2004, with an introduction by J. Picard.

\bibitem{MR1482931}
{J.~R. Norris},
\newblock \emph{Long-time behaviour of heat flow: global estimates and exact
  asymptotics},
\newblock {Arch. Rational Mech. Anal.} 140(2):161--195, 1997.

\bibitem{MR2320609}
{L.~Riahi},
\newblock \emph{Dirichlet {G}reen functions for parabolic operators with singular
  lower-order terms},
\newblock {JIPAM. J. Inequal. Pure Appl. Math.} 8(2):Article 36, 24 pp.
  (electronic), 2007.

\bibitem{MR1687500}
{R.~Schnaubelt \and J.~Voigt},
\newblock \emph{The non-autonomous {K}ato class},
\newblock {Arch. Math. (Basel)} 72(6):454--460, 1999.

\bibitem{MR1457736}
{Q.~S. Zhang},
\newblock \emph{Gaussian bounds for the fundamental solutions of {$\nabla (A\nabla
  u)+B\nabla u-u_t=0$}},
\newblock {Manuscripta Math.} 93(3):381--390, 1997.

\bibitem{MR1488344}
{Q.~S. Zhang},
\newblock \emph{On a parabolic equation with a singular lower order term. {II}. {T}he
  {G}aussian bounds},
\newblock {Indiana Univ. Math. J.} 46(3):989--1020, 1997.

\bibitem{MR1978999}
{Q.~S. Zhang},
\newblock \emph{A sharp comparison result concerning {S}chr\"odinger heat kernels},
\newblock {Bull. London Math. Soc.} 35(4):461--472, 2003.

\bibitem{MR1132313}
{Z. Zhao},
\newblock \emph{A probabilistic principle and generalized {S}chr\"odinger perturbation},
\newblock {J. Funct. Anal.}, 101(1):162--176, 1991.
\end{thebibliography}
\end{document}